\documentclass[12pt]{amsart}

\usepackage{amsfonts}
\usepackage{amsmath}
\usepackage{amsthm}

\theoremstyle{plain}
\newtheorem{theorem}{Theorem}

\newtheorem{lemma}[theorem]{Lemma}

\theoremstyle{definition}
\newtheorem{definition}[theorem]{Definition}

\theoremstyle{remark}
\newtheorem{remark}[theorem]{Remark}

\title{Equivariant structure constants for Hamiltonian-$T$-spaces}

\author{Ho-Hon Leung}

\begin{document}
\begin{abstract}
If there exists a set of canonical classes on a compact Hamiltonian-$T$-space in the sense of R. Goldin and S. Tolman, we derive some formulas for certain equivariant structure constants in terms of other equivariant structure constants and the values of canonical classes restricted to some fixed points. These formulas can be regarded as a generalization of Tymoczko's results.
\end{abstract}

\maketitle


\section{Introduction} \label{section1}
Let $T$ be a compact torus with its Lie algebra $\mathfrak{t}$ and lattice $\mathfrak{l} \subset\mathfrak{t}$. For a compact symplectic manifold $(M,\omega)$ equipped with a Hamiltonian-$T$-action, we have a \emph{moment map} $\phi\colon M\rightarrow \mathfrak{t}^\ast$, where $\mathfrak{t}^\ast$ is the dual of $\mathfrak{t}$. Then we have the following equation: \begin{equation*} \iota_{X_\xi}\omega=-d\phi^\xi,\mbox{       }\forall\xi\in\mathfrak{t}\end{equation*} where $X_\xi$ denotes the vector field on $M$ generated by the action and $\phi^\xi \colon M\to\mathbb{R}$ is defined by $\phi^\xi (x)=\langle\phi (x),\xi\rangle$. Here, $\langle .,.\rangle$ is the natural pairing of $\mathfrak{t}^\ast$ and $\mathfrak{t}$. M is called a \emph{compact Hamiltonian-$T$-space}.

$\phi^\xi$ is called \emph{the component} of the moment map $\phi$ corresponding to the chosen element $\xi\in\mathfrak{t}$. Suppose that the component of the moment map is \emph{generic}, that is, $\langle\eta,\xi\rangle\neq 0$ for each weight $\eta\in \mathfrak{l}^\ast \subset\mathfrak{t}^\ast$ in the symplectic representation $T_p M$ for every $p$ in the $T$-fixed point set $M^T$, then $\psi=\phi^\xi \colon M\to\mathbb{R}$ is a Morse function with the critical set $M^T$. Under this situation, the Morse index of $\psi$ at each $p\in M^T$ is even. Let $\lambda(p)$ be half of the index of $\psi$ at $p$. Let $\wedge_p^{-}$ be the product of all the individual weights of this representation. 

For each $p\in M^T$, the natural inclusion map $i_p\colon p \to M$ induces a map $i_p^\ast\colon H_T^\ast(M)\to H_T^\ast(p)$ in equivariant cohomology. Let $\alpha\in H_T^\ast(M)$, we use the notation $\alpha(p)$ for the image of $\alpha$ under the map $i_p^\ast$. $i_p^\ast$ is called \emph{the localization at p}. 

\begin{definition} [\cite{GoldinTolman}]
Let M be a compact Hamiltonian-$T$-space with the moment map $\phi\colon M\to\mathfrak{t}^\ast$ and let $\psi=\phi^\xi \colon M\to\mathbb{R}$ be a generic component of the moment map for some $
\xi\in\mathfrak{t}$, a cohomology class $\alpha_p \in H_T^{2\lambda(p)}(M;\mathbb{Q})$ is a canonical class at the fixed point $p$ with respect to $\psi$ if 
\begin{enumerate}
\item $\alpha_p(p)=\wedge_p^{-}$.
\item $\alpha_p(q)=0$ for all $q\in M^T \backslash\{p\}$ such that $\lambda(q)\leq\lambda(p)$.
\end{enumerate}
\end{definition} 

Canonical classes do not always exist, see Example 2.2 in \cite{GoldinTolman}. But if canonical classes exist for all $p\in M^T$, then $\{\alpha_p\}_{p\in M^T}$ form a basis of $H_T^\ast(M)$ as a module over $H_T^\ast(pt)\cong H^\ast(BT)$. 

Suppose that a set of canonical classes exists, the \emph{equivariant structure constants} for $H_T^\ast(M)$ are the elements $c_{pq}^r \in H_T^\ast(pt)$ given by the equation \begin{equation}\alpha_p \alpha_q =\sum_{r\in M^T} c_{pq}^r \alpha_r.\end{equation}

In \cite{Tymoczko}, explicit formulas for the equivariant structure constants of $H_T^\ast(\mathbb{CP}^n)$ are computed in terms of the localizations of canonical classes at various fixed points in $(\mathbb{CP}^n)^T$. This paper is concerned with the generalization of these formulas to compact Hamiltonian-$T$-spaces, under the assumption that the set of canonical classes exists. 

Given a directed graph with vertex set $V$ and edge set $E\subset V\times V$, a path from a vertex $p$ to a vertex $q$ is a $(k+1)$-tuple $r=(r_0,...,r_k)\in V^{k+1}$ so that $r_0=p,r_k=q$ and $(r_{i-1},r_i)\in E$ for all $1\leq i\leq k$. 

\begin{definition} [\cite{GoldinTolman}] \label{directedgraph}
Define an oriented graph with the vertex set $V=M^T$ and the edge set 
\begin{equation*}
E=\{(r,r')\in M^T\times M^T\mid\lambda(r')-\lambda(r)=1,\alpha_r(r')\neq 0\}.
\end{equation*}
Let $\sum_p^q$ be the set of paths from $p$ to $q$ in $(V,E)$.
\end{definition}

From now on, we call this graph a \emph{moment graph}. Note that if no path connects $p$ and $q$, i.e. $\sum_p^q$ is an empty set, then $\alpha_p(q)=0$. 

The realization of the image of a moment map as a graph has been known for a while, see \cite{GulleminZara} for example. Some useful information about the equivariant cohomology of a Hamiltonian-$T$-space can be extracted by using such a graph. 

A formula for $\alpha_p (q)$ in terms of the values of a moment map at the points in $M^T$ and the restriction of canonical classes to points of index exactly two higher was derived in \cite{GoldinTolman}. Based on their idea, more formulas are derived in \cite{Zara}. The goal of this paper is totally different: Formulas for some equivariant structure constants are written in terms of other equivariant structure constants and the restriction of canonical classes to the $T$-fixed point set $M^T$. These are the main results in Section \ref{section2}. The complexity of computations involved in our formulas depends heavily on the structure of the moment graph. In some special cases, if the structure of the moment graph is exceptionally simple, our formulas are greatly simplified. We will look at an example in Section \ref{section3}. Note that we don't make use of any extra assumption on the Hamiltonian-$T$-spaces except the existence of a set of canonical classes.

\section{Main Results} \label{section2}

Let $(M,\omega)$ be a compact Hamiltonian-$T$-space and $\psi=\phi^\xi$ be a generic component of the moment map. Assume that a set of canonical classes $\alpha_p\in H_T^{2\lambda(p)}(M;\mathbb{Q})$ exists for all $p\in M^T$, define an oriented graph $(V,E)$ as in Definition \ref{directedgraph}. In this section we compute the equivariant structure constants $c_{pq}^k$ for any $p,q$ in the vertex set of the moment graph. We do the computations following the values of $\lambda(k)$ in asecending order.

\begin{lemma} \label{lemma1}

\begin{equation*}
c_{pq}^k=0
\end{equation*}for $\lambda(k)<\lambda(p)\leq\lambda(q)$, where $p,q,k\in M^T$. 
\end{lemma}

\begin{proof}
We begin by writing the equation
\begin{equation}\label{star}
\alpha_p \alpha_q =\sum_{r\in M^T} c_{pq}^r \alpha_r.
\end{equation}
Without loss of generality, we assume that $\lambda(p)\leq\lambda(q)$. Let $t$ be an element in $M^T$ such that $\lambda(t)$ is the minimum value in the set $S=\{\lambda(x)\mid x\in M^T\}$. Since $\lambda(t)<\lambda(q)$, we have $\alpha_q (t)=0$. Localizing (\ref{star}) at $t$ gives
\begin{equation} \label{equation1}
\sum_{r\in M^T} c_{pq}^r \alpha_r(t)=0.
\end{equation}
Since $\alpha_r (t)=0,\forall r\in M^T\backslash\{t\}$, (\ref{equation1}) implies
\begin{equation*} 
c_{pq}^t \alpha_t(t)=0.
\end{equation*}
But $\alpha_t(t)\neq 0$, thus we get 
\begin{equation} \label{equationt}
c_{pq}^t=0.
\end{equation}

In the set $S=\{\lambda(x)\mid x\in M^T\}$, pick $u\in M^T\backslash\{q\}$ such that $\lambda(u)\leq\lambda(q)$ and $\lambda(u)$ attains the minimum value in the set $S\backslash\{\lambda(t)\}$, where $t\in M^T$ still satisfies the same property as above that $\lambda(t)=\mbox{min}_{x\in M^T}\lambda(x)$. Then $\alpha_q(u)=0$ and hence localizing (\ref{star}) at $u$ gives
\begin{equation} \label{equation2}
\sum_{r\in M^T} c_{pq}^r \alpha_r(u)=0.
\end{equation}
But we know that $\alpha_r(u)=0$ when $u\neq r$ and $\lambda(u)\leq\lambda(r)$.
Also, $c_{pq}^t=0$ by (\ref{equationt}). Hence (\ref{equation2}) gives
\begin{equation*}
c_{pq}^u \alpha_u(u)=0.
\end{equation*}
Since $\alpha_u(u)\neq 0$, we have 
\begin{equation}
c_{pq}^u =0.
\end{equation}

By using the same method inductively on the set of values in $S$ which are smaller than $\lambda(p)$, we conclude that 
\begin{equation}
c_{pq}^k =0
\end{equation}
for all $k\in M^T$ such that $\lambda(k)<\lambda(p)\leq\lambda(q)$.
\end{proof}

\begin{lemma} \label{lemma2}
\begin{equation*}
c_{pq}^p=0
\end{equation*}for $\lambda(p)\leq\lambda(q)$, where $p,q\in M^T$.
\end{lemma}

\begin{proof}
Note that $\alpha_q(p)=0$. Localizing (\ref{star}) at $p$ gives
\begin{equation}\label{equationp}
\sum_{r\in M^T} c_{pq}^r \alpha_r(p)=0.
\end{equation}For all $k\in M^T\backslash\{p\}$ such that $\lambda(k)<\lambda(p)$, $c_{pq}^k=0$ by Lemma \ref{lemma1}. For all $k'\in M^T\backslash\{p\}$ such that $\lambda(k')\geq\lambda(p)$, $\alpha_{k'}(p)=0$. Hence by (\ref{equationp}), 
\begin{equation*}
c_{pq}^p \alpha_p(p)=0.
\end{equation*}Since $\alpha_p(p)\neq 0$, we have
\begin{equation}
c_{pq}^p=0.
\end{equation}
\end{proof}

\begin{lemma} \label{lemma3}
\begin{equation*}
c_{pq}^k=0
\end{equation*}
for $k\in M^T\backslash\{p,q\}$ such that $\lambda(p)\leq\lambda(k)\leq\lambda(q)$.
\end{lemma}

\begin{proof}
Note that $\alpha_q(k)=0$. Localizing (\ref{star}) at $k$ gives
\begin{equation} \label{equation10}
\sum_{r\in M^T} c_{pq}^r \alpha_r(k)=0.
\end{equation}
If $\lambda(p)=\lambda(k)$, then $c_{pq}^u=0$ for $u\in M^T$ such that $\lambda(u)<\lambda(p)=\lambda(k)$ by Lemma $\ref{lemma1}$. And $\alpha_{k'}(k)=0$ for $k' \in M^T\backslash\{k\}$ such that $\lambda(k')\geq\lambda(k)$. (\ref{equation10}) becomes 
\begin{equation}
c_{pq}^k \alpha_k(k)=0.
\end{equation}Since $\alpha_k (k)\neq 0$, we get $c_{pq}^k$=0. By using the same localization method inductively on the set $S'\subset S=\{\lambda(x)\mid x\in M^T\}$ that contains all values between $\lambda(p)$ and $\lambda(q)$, we get the result.
\end{proof}

\begin{lemma} \label{lemma4}
\begin{equation*}
c_{pq}^q=\alpha_p(q)
\end{equation*}for $\lambda(p)\leq\lambda(q)$, where $p,q\in M^T$.
\end{lemma}

\begin{proof}
Localizing (\ref{star}) at $q$ gives 
\begin{equation}\label{equation11}
\alpha_p (q)\alpha_q(q)=\sum_{r\in M^T} c_{pq}^r \alpha_r(q).
\end{equation}By Lemma \ref{lemma1}, \ref{lemma2} and \ref{lemma3}, $c_{pq}^k=0$ for all $k\in M^T\backslash\{q\}$ such that $\lambda(k)\leq\lambda(q)$. And $\alpha_{k'}(q)=0$ for all $k'\in M^T$ such that $\lambda(k')>\lambda(q)$. Hence by (\ref{equation11}),
\begin{equation}
\alpha_p(q)\alpha_q(q)=c_{pq}^q \alpha_q(q).
\end{equation}Then we divide both sides by $\alpha_q(q)$, which is non-zero, to get the desired result.
\end{proof}

Next, we will consider the equivariant structure constants $c_{pq}^z$ such that $\lambda(z)=1+\lambda(q)$. 

\begin{theorem} \label{theoremz}
By the same notations and assumptions as in Lemma \ref{lemma1}, 
\begin{equation*}
c_{pq}^z=\frac{\alpha_q(z)}{\alpha_z(z)}(\alpha_p(z)-\alpha_p(q))
\end{equation*}where $\lambda(z)=1+\lambda(q)$.
\end{theorem}

\begin{proof}
Let $z\in M^T$ such that $\lambda(z)=1+\lambda(q)$. Localizing (\ref{star}) at $z$ gives
\begin{eqnarray}
\alpha_p(z)\alpha_q(z)&=&\sum_{r\in M^T} c_{pq}^r \alpha_r(z)\nonumber\\
&=&c_{pq}^q \alpha_q(z)+c_{pq}^z \alpha_z(z).\label{equationz}
\end{eqnarray}The second equality holds because $c_{pq}^k=0$ for $k\in M^T\backslash\{q\}$ such that $\lambda(k)\leq\lambda(q)$ by Lemma \ref{lemma1}, \ref{lemma2} and \ref{lemma3}. Also, $\alpha_{k'}(z)=0$ for all $k'\in M^T\backslash\{z\}$ such that $\lambda(k')\geq\lambda(z)$. By (\ref{equationz}), 
\begin{equation}\label{equationz2}
c_{pq}^z=\frac{\alpha_p(z)\alpha_q(z)-c_{pq}^q \alpha_q(z)}{\alpha_z(z)}=\frac{\alpha_p(z)\alpha_q(z)-\alpha_p(q) \alpha_q(z)}{\alpha_z(z)}=\frac{\alpha_q(z)}{\alpha_z(z)}(\alpha_p(z)-\alpha_p(q)).
\end{equation}
\end{proof}

\begin{remark} \label{remark1}
We note that if $(q,z)\notin E$, which means that there is no edge connecting $q$ and $z$ in the moment graph, then $\alpha_q(z)=0$ and hence $c_{pq}^z=0$ by (\ref{equationz2}). 
\end{remark}

We will then consider the equivariant structure constants $c_{pq}^y$ such that $\lambda(y)=2+\lambda(q)$.

\begin{definition}
In the directed graph defined in Definition \ref{directedgraph}, define the \emph{negative valency}, $V_p^{-}$, at $p\in V$ by
\begin{equation*}
V_p^{-} =\{v\in V\mid (v,p)\in E\}.
\end{equation*}
Define the positive valency, $V_p^{+}$, at $p\in V$ by
\begin{equation*}
V_p^{+}=\{v\in V\mid (p,v)\in E\}
\end{equation*}
and let $| V_p|$ be the number of elements in $V_p$.
\end{definition}

\begin{definition} \label{shifting}
Let the rank of the torus $T$ be $n$. Let $I\subset\mathbb{Q}[t_0, t_1,..., t_n]$ denote the subring generated by $\alpha_p(q)$ for all $p,q\in M^T$. Define a \emph{shifting operator} $\mathfrak{s}_a^b\colon I\rightarrow I$ by
\begin{equation*}
\mathfrak{s}_a^b(\alpha_p(a))=\alpha_p(b)
\end{equation*}for any $p\in M^T$.
\end{definition}

Note that the definition of $\mathfrak{s}_a^b$ can be extended to the ring of fractions of $I$. Now we are in the right place to state our next result.

\begin{theorem} \label{theoremy}
By the same notations and assumptions as in Lemma \ref{lemma1},
\begin{equation*} 
c_{pq}^y=\sum_{i=1}^{|V_y^{-}|} \frac{\alpha_{z_i}(y)}{\alpha_y(y)}(\frac{1}{| V_y^{-}|}\mathfrak{s}_{z_i}^y c_{pq}^{z_i}-c_{pq}^{z_i})
\end{equation*}where $y\in M^T$ such that $\lambda(y)=2+\lambda(q)$ and $z_i$ are the elements in $V_y^-$, for $i=1,2,...,|V_y^-|$.
\end{theorem}

\begin{proof}
Let $y\in M^T$ such that $\lambda(y)=2+\lambda(q)$. Localizing (\ref{star}) at $y$ gives
\begin{equation}\label{equation16}
\alpha_p(y)\alpha_q(y)=\sum_{r\in M^T}c_{pq}^r \alpha_r(y).
\end{equation}Note that $c_{pq}^k=0$ if $\lambda(k)\leq\lambda(q)$ and $k\neq q$. Also, $\alpha_{k'}(y)=0$ if $\lambda(k')\geq\lambda(y)$ and $k'\neq y$. For $z\in M^T$ such that $\lambda(z)=1+\lambda(q)$ but $z\notin V_y^-$, $\alpha_z(y)=0$. Hence, (\ref{equation16}) is simplified as
\begin{eqnarray}
\alpha_p(y)\alpha_q(y)&=& c_{pq}^q \alpha_q(y)+\sum_{z\in V_y^{-}} c_{pq}^z \alpha_z(y)+c_{pq}^y \alpha_y(y)\nonumber\\
&=& \alpha_p(q) \alpha_q(y)+\sum_{z\in V_y^{-}} c_{pq}^z \alpha_z(y)+c_{pq}^y \alpha_y(y). \label{equation17}
\end{eqnarray}By rearranging terms in (\ref{equation17}), we get
\begin{equation} \label{equation18}
c_{pq}^y =\frac{\alpha_p(y)\alpha_q(y)-\alpha_p(q)\alpha_q(y)}{\alpha_y(y)}-\frac{\sum_{z\in V_y^{-}}c_{pq}^z \alpha_z(y)}{\alpha_y(y)}.
\end{equation}Denote the elements in $V_y^-$ by $z_1, z_2,..., z_{|V_y^-|}$. By Theorem \ref{theoremz},
\begin{equation*}
c_{pq}^{z_i}=\frac{\alpha_q(z_i)}{\alpha_{z_i}(z_i)}(\alpha_p(z_i)-\alpha_p(q))
\end{equation*} for all $z_i$ in $V_y^-$. By the shifting operators defined in Definition \ref{shifting}, we have
\begin{equation} \label{equation19}
\mathfrak{s}_{z_i}^y  c_{pq}^{z_i}=\frac{\alpha_q(y)}{\alpha_{z_i}(y)}(\alpha_p(y)-\alpha_p(q)) .
\end{equation}By (\ref{equation18}), we have
\begin{equation}
c_{pq}^y =\frac{\alpha_{z_i}(y)}{\alpha_y(y)}\mathfrak{s}_{z_i}^y c_{pq}^{z_i} -\frac{\sum_{i=1}^{| V_y^{-}|} c_{pq}^{z_i} \alpha_{z_i}(y)}{\alpha_y(y)}=\frac{\alpha_{z_i}(y)}{\alpha_y(y)}(\mathfrak{s}_{z_i}^y c_{pq}^{z_i} -\sum_{i=1}^{|V_y^-|} c_{pq}^{z_i} )
\end{equation} for each $z_i\in V_y^-$. Adding all these $| V_y^{-}|$ equations together, and then dividing the sum by $| V_y^{-}|$, we have
\begin{equation}  \label{equation21}
c_{pq}^y=\sum_{i=1}^{| V_y^{-}|} \frac{\alpha_{z_i}(y)}{\alpha_y(y)}(\frac{1}{| V_y^{-}|}\mathfrak{s}_{z_i}^y c_{pq}^{z_i}-c_{pq}^{z_i}).
\end{equation}
\end{proof}

\begin{remark} \label{remark2}
If there is no path connecting $q$ and $y$ in $(V,E)$ when $\lambda(y)-\lambda(q)=2$, then $\alpha_q(y)=0$. Under this situation, for $z\in V_y^{-}$, $\alpha_z(y)\neq 0$ but $c_{pq}^z=0$ (see Remark \ref{remark1}) since there does not exist any path connecting $q$ and $z$ in $(V,E)$. Thus, $c_{pq}^z \alpha_z(y)=0$ for all $z\in V_y^-$.  By (\ref{equation18}), we can conclude that $c_{pq}^y =0$ if $\sum_q^y$ is an empty set. 
\end{remark}

Finally, we will consider the equivariant structure constants $c_{pq}^x$ where $\lambda(x)=3+\lambda(q)$.

\begin{theorem} \label{theoremx}
By the same notations and assumptions as in Lemma \ref{lemma1},
\begin{equation*}
c_{pq}^x=\sum_{y\in V_x^-}\frac{\alpha_y(x)}{\alpha_x(x)}(\frac{1}{|V_x^-|}\mathfrak{s}_y^x c_{pq}^y-c_{pq}^y)+\sum_{z\in M^T\backslash (\{q,x\}\cup V_x^-)}\frac{| V_z^+| -| V_x^-|}{|V_x^-|}\frac{\alpha_z(x)}{\alpha_x(x)}c_{pq}^z
\end{equation*}where $x\in M^T$ such that $\lambda(x)=3+\lambda(q)$.
\end{theorem}

\begin{proof}
For $x\in M^T$ such that $\lambda(x)=3+\lambda(q)$, localizing (\ref{star}) at $x$ gives
\begin{equation} \label{equation22}
\alpha_p(x)\alpha_q(x)=\sum_{r\in M^T}c_{pq}^r \alpha_r(x).
\end{equation}Note that $c_{pq}^k=0$ if $\lambda(k)\leq\lambda(q)$ and $k\neq q$. $\alpha_{k'}(x)=0$ if $\lambda(k')\geq\lambda(x)$ and $k' \neq x$. For $y\in M^T$ such that $\lambda(y)=2+\lambda(q)$, the term $c_{pq}^y \alpha_y(x)$ is non-zero only if $y\in V_x^-$. By (\ref{equation22}), we have
\begin{eqnarray}
\alpha_p(x) \alpha_q(x)&=&c_{pq}^q \alpha_q(x)+\sum_{z\in M^T\backslash(\{q,x\}\cup V_x^{-})}c_{pq}^z \alpha_z(x) +\sum_{y\in V_x^{-}} c_{pq}^y \alpha_y(x)+c_{pq}^x \alpha_x(x)\nonumber\\
&=&\alpha_p(q) \alpha_q(x)+\sum_{z\in M^T\backslash(\{q,x\}\cup V_x^{-})}c_{pq}^z \alpha_z(x) +\sum_{y\in V_x^{-}} c_{pq}^y \alpha_y(x)+c_{pq}^x \alpha_x(x). \nonumber
\end{eqnarray}The terms included in the second term on the right hand side can be non-zero only when $z\in V_q^{+}$ and $\sum_z^x$ is a non-empty set. Hence, by rearranging the terms, we get
\begin{equation}\label{equation23}
c_{pq}^x=\frac{\alpha_p(x)\alpha_q(x)-\alpha_p(q)\alpha_q(x)-\sum_{z\in V_q^{+}}c_{pq}^z \alpha_z(x)}{\alpha_x(x)} -\frac{\sum_{y\in V_x^{-}}c_{pq}^y \alpha_y(x)}{\alpha_x(x)}.
\end{equation}For $y\in V_x^{-}$, by (\ref{equation18}),
\begin{equation} \label{equation24}
\frac{\alpha_y(x)}{\alpha_x(x)}\mathfrak{s}_y^x c_{pq}^y=\frac{\alpha_p(x)\alpha_q(x)-\alpha_p(q)\alpha_q(x)-\sum_{z\in V_y^{-}}c_{pq}^z \alpha_z(x)}{\alpha_x(x)}.
\end{equation}The last term in the numerator on the right side of (\ref{equation24}) can be non-zero only when $z\in V_q^{+}$ and $\sum_z^x$ is non-empty. By (\ref{equation23}) and (\ref{equation24}), for each $y\in V_x^{-}$,
\begin{equation} \label{equation25}
c_{pq}^x=\frac{\alpha_y(x)}{\alpha_x(x)} \mathfrak{s}_y^x c_{pq}^y-\sum_{z\in V_q^+\backslash V_y^-}\frac{\alpha_z(x)}{\alpha_x(x)}c_{pq}^z - \sum_{y\in V_x^-}\frac{\alpha_y(x)}{\alpha_x(x)}c_{pq}^y.
\end{equation}Before adding up the equations (\ref{equation25}) for each $y\in V_x^-$, let us focus on the second term on the right side of (\ref{equation25}). Since we are only interested in those non-zero terms, we only have to take care of all the terms for those $z\in V_q^+$ when there is at least one path connecting $q$, $z$ and $x$. The simplest case is that $| V_z^+|=1$ for all $z\in V_q^+$. That is, $z$ is only connected to one and only one $y\in V_x^-$. In this case, the sets $V_y^-$ for each $y\in V_x^-$ are all disjoint. It implies that $V_q^+$ is a disjoint union of $V_y^-$ for each $y\in V_x^-$. Then by adding (\ref{equation25}) for all $y\in V_x^-$, we get
\begin{equation} \label{equation26}
|V_x^-| c_{pq}^x=\sum_{y\in V_x^-}\frac{\alpha_y(x)}{\alpha_x(x)}\mathfrak{s}_y^x c_{pq}^y-\sum_{y\in V_x^-} \sum_{z\in V_q^+\backslash V_y^-}\frac{\alpha_z(x)}{\alpha_x(x)}c_{pq}^z-| V_x^-| \sum_{y\in V_x^-}\frac{\alpha_y(x)}{\alpha_x(x)}c_{pq}^y.
\end{equation}For the second term on the right side, we have
\begin{eqnarray}
\sum_{y\in V_x^-} \sum_{z\in V_q^+\backslash V_y^-}\frac{\alpha_z(x)}{\alpha_x(x)}c_{pq}^z &=&
\sum_{y\in V_x^-}(\sum_{z\in V_q^+}\frac{\alpha_z(x)}{\alpha_x(x)}c_{pq}^z-\sum_{z\in V_y^-}\frac{\alpha_z(x)}{\alpha_x(x)}c_{pq}^z)\nonumber\\
&=&|V_x^-|\sum_{z\in V_q^+}\frac{\alpha_z(x)}{\alpha_x(x)}c_{pq}^z-\sum_{y\in V_x^-}\sum_{z\in V_y^-}\frac{\alpha_z(x)}{\alpha_x(x)}c_{pq}^z\nonumber\\
&=&|V_x^-|\sum_{z\in V_q^+}\frac{\alpha_z(x)}{\alpha_x(x)}c_{pq}^z-\sum_{z\in V_q^+}\frac{\alpha_z(x)}{\alpha_x(x)}c_{pq}^z\nonumber\\
&=& (|V_x^-|-1)\sum_{z\in V_q^+}\frac{\alpha_z(x)}{\alpha_x(x)}c_{pq}^z.\label{equation27}
\end{eqnarray}Substitute (\ref{equation27}) into (\ref{equation26}) to get
\begin{equation} \label{equation28}
| V_x^-| c_{pq}^x=\sum_{y\in V_x^-}\frac{\alpha_y(x)}{\alpha_x(x)}\mathfrak{s}_y^x c_{pq}^y+(1-| V_x^-|)\sum_{z\in V_q^+}\frac{\alpha_z(x)}{\alpha_x(x)}c_{pq}^z-| V_x^-| \sum_{y\in V_x^-}\frac{\alpha_y(x)}{\alpha_x(x)}c_{pq}^y.
\end{equation}Dividing (\ref{equation28}) by $| V_x^-|$, we get
\begin{equation}
c_{pq}^x=\sum_{y\in V_x^-}\frac{\alpha_y(x)}{\alpha_x(x)}(\frac{1}{|V_x^-|}\mathfrak{s}_y^x c_{pq}^y-c_{pq}^y)+\frac{1 -| V_x^-|}{|V_x^-|}\sum_{z\in M^T\backslash (\{q,x\}\cup V_x^-)}\frac{\alpha_z(x)}{\alpha_x(x)}c_{pq}^z.
\end{equation}which is our desired formula (when $|V_z^+|=1$ for all $z\in V_q^+$).

More generally, if $|V_z^+|>1$ for some $z\in V_q^+$, we have to take care of those `excessive edges' coming out of each $z\in V_q^+$. For each of these `excessive edges', we have an extra term $-\alpha_z(x)c_{pq}^z/ \alpha_x(x)$ in (\ref{equation27}). The number of these `excessive edges' for each $z\in V_q^+$ is $| V_z^+|-1$. It means that we have an extra term $-(|V_z^+|-1)\alpha_z(x)c_{pq}^z/ \alpha_x(x)$. Hence (\ref{equation27}) becomes
\begin{eqnarray}
\sum_{y\in V_x^-} \sum_{z\in V_q^+\backslash V_y^-}\frac{\alpha_z(x)}{\alpha_x(x)}c_{pq}^z 
&=& \sum_{z\in V_q^+}[(|V_x^-|-1)\frac{\alpha_z(x)}{\alpha_x(x)}c_{pq}^z-(|V_z^+|-1)\frac{\alpha_z(x)}{\alpha_x(x)}c_{pq}^z]\nonumber\\
&=& \sum_{z\in V_q^+}(|V_x^-|-|V_z^+|)\frac{\alpha_z(x)}{\alpha_x(x)}c_{pq}^z.\label{equation30}
\end{eqnarray}Substitute (\ref{equation30}) into (\ref{equation26}) and divide (\ref{equation26}) by $|V_x^-|$ to get the desired formula.
\end{proof}

\section{An example: Complex Projective Space} \label{section3}
A simple example for a compact Hamiltonian-$T$-space is $\mathbb{CP}^n$. The $T$-action is defined by $(t_0,...,t_n).[z_0,...,z_n]=[t_0 z_0,...,t_n z_n]$. The moment polytope is the $n$-simplex. By suitably choosing a generic component of the moment map, we get the Morse function. There are $n+1$ vertices in the moment graph. We label the vertices by $p_0, p_1,..., p_n$ in the ascending order of their indices. $|V_{p_i}^+|=|V_{p_i}^-|=1$ for all $i$ except $i=0$ and $i=n$. By Lemma 3.2 in \cite{Tymoczko}, the classes $\alpha_{p_i}$ defined by $\alpha_{p_i}(p_k)=\prod_{j=0}^{i-1}(t_j-t_k)$ for $i\leq k,i=1,...,n$ can be used as the set of canonical classes for $H_T^\ast(\mathbb{CP}^n)$. Thus, we have $\alpha_{p_{k-1}}(p_k)/ \alpha_{p_k}(p_k)=1/(t_{k-1}-t_k)$. By Theorem \ref{theoremz}, \ref{theoremy} and \ref{theoremx}, we have \begin{equation} \label{equation31}
c_{p_i p_j}^{p_k}=\frac{\mathfrak{s}_{p_{k-1}}^{p_k} c_{p_i p_j}^{p_{k-1}}-c_{p_i p_j}^{p_{k-1}}}{t_{k-1}-t_k}
\end{equation} when $\lambda(p_i)\leq\lambda(p_j)$ and $\lambda(p_k)-\lambda(p_j)=1, 2, 3$.

More generally, for $\lambda(p_k)-\lambda(p_j)>3$, it is straightforward to check that (\ref{equation31}) still holds by the localization method used in the proofs of Theorem \ref{theoremz}, \ref{theoremy} and \ref{theoremx}. Hence we have obtained Theorem 4.1 in \cite{Tymoczko} as a special case of our results. 

\begin{remark}
The right side of (\ref{equation31}) is the same as $\partial_{k-1}c_{p_i p_j}^{p_{k-1}}$ where $\partial_{k-1}$ is the \emph{divided difference operator} defined in \cite{Tymoczko}. Divided difference operators are also defined in Kasparov's equivariant $KK$-theory. For the definitions and some interesting applications of divided difference operators in $K$-theory and $KK$-theory, see \cite{Leung1} and \cite{Leung2}.
\end{remark}

\end{document}